\documentclass{amsart}

\usepackage{amsmath,latexsym,amssymb,amsthm,graphics}
\usepackage[all]{xy}

\newtheorem{theorem}{Theorem}
\newtheorem{lemma}{Lemma}

\newtheorem*{maincorollary}{Corollary}

\newcommand{\tn}[1]{\textnormal{#1}}

\newcommand{\pa}[1]{\left(#1\right)}
\newcommand{\br}[1]{\left[#1\right]}
\newcommand{\cbr}[1]{\left\{#1\right\}}
\newcommand{\fg}[1]{\pi_{1}\pa{#1}}
\newcommand{\tfg}[1]{\pi_{1}^{\tn{top}}\pa{#1}}

\begin{document}
\title{Discreteness and Homogeneity of the Topological Fundamental Group}
\author[J. Calcut]{Jack S. Calcut}
\address{Department of Mathematics\\
					Michigan State University\\
					East Lansing, MI 48824-1027}
\email{jack@math.msu.edu}
\urladdr{http://www.math.msu.edu/~jack/}

\author[J. McCarthy]{John D. McCarthy}
\address{Department of Mathematics\\
					Michigan State University\\
					East Lansing, MI 48824-1027}
\email{mccarthy@math.msu.edu}
\urladdr{http://www.math.msu.edu/~mccarthy/}
       
\subjclass[2000]{}
\date{March 25, 2009}   

\begin{abstract}
For a locally path connected topological space, the topological fundamental group is discrete if and only if the space is semilocally simply-connected. While functoriality of the topological fundamental group for arbitrary topological spaces remains an open question, the topological fundamental group is always a homogeneous space.
\end{abstract}

\maketitle

\section{Introduction}

The concept of a natural topology for the fundamental group appears to have originated with Hurewicz~\cite{wH35} in 1935. It received further attention by Dugundji~\cite{jD50} in 1950 and by Biss~\cite{dB02}, Fabel~\cite{pF08,pF07,pF05,pF06}, and others more recently. The purpose of this note is to prove the following folklore theorem.

\begin{theorem}\label{discretetheorem}
Let $X$ be a locally path connected topological space. The topological fundamental group $\tfg{X}$ is discrete if and only if $X$ is semilocally simply-connected.
\end{theorem}

Theorem~5.1 of~\cite{dB02} is Theorem~\ref{discretetheorem} without the hypothesis of local path connectedness. However a counterexample of Fabel~\cite{pF07} shows that this stronger result is false. Fabel~\cite{pF07} also proves a weaker version of Theorem~\ref{discretetheorem} assuming that $X$ is locally path connected and a metric space. In this note we remove the metric hypothesis.\\

Our proof proceeds from first topological principles making no use of rigid covering fibrations~\cite{dB02} nor even of classical covering spaces. We make no use of the functoriality of the topological fundamental group, a property which was also a main result in~\cite[Cor.~3.4]{dB02} but in fact is unproven~\cite[pp.~188--189]{pF06}. Beware that the misstep in the proof of~\cite[Prop.~3.1]{dB02}, namely the assumption that the product of quotient maps is a quotient map, is repeated in~\cite[Thm.~2.1]{hG08}.\\

In general the homeomorphism type of the topological fundamental group depends on a choice of basepoint. We say that $\tfg{X}$ is \emph{discrete} without reference to basepoint provided $\tfg{X,x}$ is discrete for each $x\in X$. If $x$ and $y$ are connected by a path in $X$, then $\tfg{X,x}$ and $\tfg{X,y}$ are homeomorphic. This fact was proved in~\cite[Prop.~3.2]{dB02} and a detailed proof is in Section~\ref{basepoint} below for completeness. Theorem~\ref{discretetheorem} now immediately implies the following.

\begin{maincorollary}
Let $X$ be a path connected and locally path connected topological space. The topological fundamental group $\tfg{X,x}$ is discrete for some $x\in X$ if and only if $X$ is semilocally simply-connected.
\end{maincorollary}

As mentioned above it is open whether $\pi_{1}^{\tn{top}}$ is a functor from the category of pointed topological spaces to the category of topological groups. The unsettled question is whether multiplication
\[
\xymatrix@R=0pt{
	\tfg{X,x}\times\tfg{X,x}	\ar[r]^-{\mu}		&		\tfg{X,x}\\
	\pa{\br{f},\br{g}}				\ar@{|-{>}}[r]	&		\br{f}\cdot\br{g}}
\]
is continuous. By Theorem~\ref{discretetheorem}, if $X$ is locally path connected and semilocally simply-connected, then $\tfg{X,x}$, and hence $\tfg{X,x}\times\tfg{X,x}$, is discrete and so $\mu$ is trivially continuous. Continuity of $\mu$ in general remains an interesting question.\\

Lemma~\ref{leftrightlemma} below shows that if $(X,x)$ is an arbitrary pointed topological space, then left and right multiplication by any fixed element in $\tfg{X,x}$ are continuous self maps of $\tfg{X,x}$. Therefore $\tfg{X,x}$ acts on itself by left and right translation as a group of self homeomorphisms. Clearly these actions are both transitive. Thus we obtain the following result.

\begin{theorem}
If $(X,x)$ is a pointed topological space, then $\tfg{X,x}$ is a homogeneous space.
\end{theorem}

This note is organized as follows. Section~\ref{defs} contains definitions and conventions, Section~\ref{proof} proves two lemmas and Theorem~\ref{discretetheorem}, Section~\ref{basepoint} addresses change of basepoint, and Section~\ref{translation} shows left and right translation are homeomorphisms.

\section{Definitions and Conventions}\label{defs}

By convention, neighborhoods are open. Unless stated otherwise, homomorphisms are inclusion induced.\\

Let $X$ be a topological space and $x\in X$. A neighborhood $U$ of $x$ is \emph{relatively inessential} (in $X$) provided $\fg{U,x}\rightarrow\fg{X,x}$ is trivial. $X$ is \emph{semilocally simply-connected} at $x$ provided there exists a relatively inessential neighborhood $U$ of $x$. $X$ is \emph{semilocally simply-connected} provided it is so at each $x\in X$. A neighborhood $U$ of $x$ is \emph{strongly relatively inessential} (in $X$) provided $\fg{U,y}\rightarrow\fg{X,y}$ is trivial for every $y\in U$.\\

The fundamental group is a functor from the category of pointed topological spaces to the category of groups. Consequently if $A$ and $B$ are any subsets of $X$ such that $x\in A\subset B\subset X$ and $\fg{B,x}\rightarrow\fg{X,x}$ is trivial, then $\fg{A,x}\rightarrow\fg{X,x}$ is trivial as well. This observation justifies the convention that neighborhoods are open.\\

If $X$ is locally path connected and semilocally simply-connected, then each $x\in X$ has a path connected relatively inessential neighborhood $U$. Such a $U$ is necessarily a strongly relatively inessential neighborhood of $x$ as the reader may verify (see for instance~\cite[Ex.~5 p.~330]{jM75}).\\

Let $\pa{X,x}$ be a pointed topological space and let $I=\br{0,1}\subset\mathbb{R}$. The space
\[
	C_{x}(X)=\cbr{f:\pa{I,\partial I}\rightarrow\pa{X,x}\mid f  \text{ is continuous}}
\]
is endowed with the compact-open topology. The function
\[
\xymatrix@R=0pt{
	C_{x}(X)	\ar[r]^-{q}				&		\fg{X,x}\\
	f					\ar@{|-{>}}[r]		&		\br{f}}
\]
is surjective so $\fg{X,x}$ inherits the quotient topology and one writes $\tfg{X,x}$ for the resulting \emph{topological fundamental group}. Let $e_{x}\in C_{x}(X)$ denote the constant map. If $f\in C_{x}(X)$, then $f^{-1}$ denotes the path defined by $f^{-1}(t)=f(1-t)$.\\

\section{Proof of Theorem~\ref{discretetheorem}}\label{proof}

We prove two lemmas and then Theorem~\ref{discretetheorem}.

\begin{lemma}\label{lem1}
Let $\pa{X,x}$ be a pointed topological space. If $\cbr{\br{e_{x}}}$ is open in $\tfg{X,x}$, then $x$ has a relatively inessential neighborhood in $X$.
\end{lemma}

\begin{proof}
The quotient map $q$ is continuous and $\cbr{\br{e_{x}}}\subset\tfg{X,x}$ is open, so $q^{-1}\pa{\br{e_{x}}}=\br{e_{x}}$ is open in $C_{x}\pa{X}$. Therefore $e_{x}$ has a basic open neighborhood
\begin{equation}\label{defV}
	e_{x}\in	V=\bigcap_{n=1}^{N}V\pa{K_{n},U_{n}}	\subset	\br{e_{x}}	\subset C_{x}\pa{X}
\end{equation}
where each $K_{n}\subset I$ is compact, each $U_{n}\subset X$ is open, and each $V\pa{K_{n},U_{n}}$ is a subbasic open set for the compact-open topology on $C_{x}\pa{X}$. We will show that
\[
	U=\bigcap_{n=1}^{N}U_{n}
\]
is a relatively inessential neighborhood of $x$ in $X$. Clearly $U$ is open in $X$ and, by~\eqref{defV}, $x\in U$. Finally, let $f:\pa{I,\partial I}\rightarrow\pa{U,x}$. For each $1\leq n\leq N$ we have
\[
	f\pa{K_{n}}	\subset	U	\subset	U_{n}.
\]
Thus $f\in\br{e_{x}}$ by~\eqref{defV} and so $\br{f}=\br{e_{x}}$ is trivial in $\fg{X,x}$.
\end{proof}

\begin{lemma}\label{lem2}
Let $\pa{X,x}$ be a pointed topological space and let $f\in C_{x}(X)$. If $X$ is locally path connected and semilocally simply-connected, then $\cbr{\br{f}}$ is open in $\tfg{X,x}$.
\end{lemma}

\begin{proof}
As $q$ is a quotient map, we must show that $q^{-1}\pa{\br{f}}=\br{f}$ is open in $C_{x}(X)$. So let $g\in\br{f}$. For each $t\in I$ let $U_{t}$ be a path connected relatively inessential neighborhood of $g(t)$ in $X$. The sets $g^{-1}\pa{U_{t}}$, $t\in I$, form an open cover of $I$. Let $\lambda>0$ be a Lebesgue number for this cover. Choose $N\in\mathbb{N}$ so that $1/N<\lambda$. For each $1\leq n\leq N$ let
\[
	I_{n}=\br{\frac{n-1}{N},\frac{n}{N}}\subset	I.
\]
Reindex the $U_{t}$s so that
\[
	g\pa{I_{n}}	\subset U_{n}	\text{ for each $1\leq n\leq N$}.
\]
The $U_{n}$s are not necessarily distinct, nor does the proof require this condition. For each $1\leq n\leq N$ let $W_{n}$ denote the path component of $U_{n}\cap U_{n+1}$ containing $g\pa{n/N}$, so
\begin{equation}\label{gnWn}
	g\pa{\frac{n}{N}}\in W_{n}	\subset \pa{U_{n}\cap U_{n+1}}	\subset	X.
\end{equation}

Consider the basic open set
\begin{equation}\label{defnewV}
	V=\pa{\bigcap_{n=1}^{N}V\pa{I_{n},U_{n}}}	\cap	\pa{\bigcap_{n=1}^{N-1}V\pa{\cbr{\frac{n}{N}},W_{n}}}	\subset C_{x}(X).
\end{equation}
By construction, $g\in V$. It remains to show that $V\subset\br{f}$. So let $h\in V$. As $\br{g}=\br{f}$, it suffices to show that $\br{h}=\br{g}$.\\

By~\eqref{defnewV} we have
\begin{align}
	h\pa{I_{n}}				&\subset	U_{n}	\quad	\text{for each } 1\leq n\leq N \text{ and}\notag\\
	h\pa{\frac{n}{N}}	&\in			W_{n}	\quad	\text{for each } 1\leq n\leq N-1.\label{hnWn}
\end{align}
For each $1\leq n\leq N-1$ let $\gamma_{n}:I\rightarrow W_{n}$ be a continuous path such that
\begin{align*}
	\gamma_{n}(0)&=h\pa{\frac{n}{N}}	\quad	\text{and}\\
	\gamma_{n}(1)&=g\pa{\frac{n}{N}},
\end{align*}
which exists by~\eqref{gnWn} and~\eqref{hnWn}. Let $\gamma_{0}=e_{x}$ and $\gamma_{N}=e_{x}$. For each $1\leq n\leq N$ define
\[
\xymatrix@R=0pt{
	I		\ar[r]^-{s_{n}}	&	I_{n}	\\
	t		\ar@{|-{>}}[r]		&	\frac{1}{N}t+\frac{n-1}{N}}
\]
and let
\begin{align*}
	g_{n}&=g\circ s_{n}	\quad \text{and}\\
	h_{n}&=h\circ s_{n}.
\end{align*}
So $g_{n}$ and $h_{n}$ are affine reparameterizations of $\left.g\right|_{I_{n}}$ and $\left.h\right|_{I_{n}}$ respectively.
For each $1\leq n\leq N$
\[
	\delta_{n}=g_{n}\ast\gamma_{n}^{-1}\ast h_{n}^{-1}\ast\gamma_{n-1}
\]
is a loop in $U_{n}$ based at $g_{n}(0)$ (see Figure~\ref{loop}).
\begin{figure}[h!]
	\centerline{\includegraphics{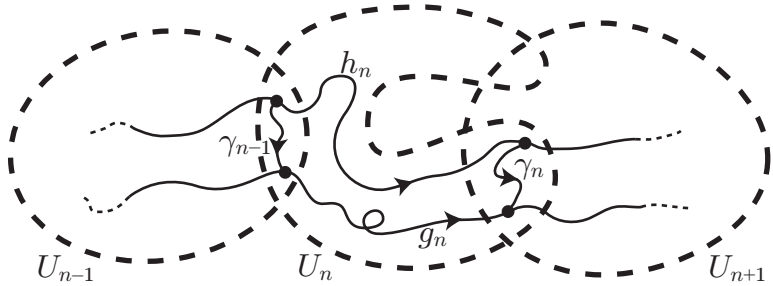}}
	\caption{Loop $\delta_{n}=g_{n}\ast\gamma_{n}^{-1}\ast h_{n}^{-1}\ast\gamma_{n-1}$ in $U_{n}$ based at $g_{n}(0)$.}
	\label{loop}
\end{figure}
As $U_{n}$ is a strongly relatively inessential neighborhood, $\br{\delta_{n}}=1\in\fg{X,g_{n}(0)}$. Therefore $g_{n}$ and $\gamma_{n-1}^{-1}\ast h_{n}\ast \gamma_{n}$ are path homotopic. In $\fg{X,x}$ we have
\begin{align*}
	\br{h}	&=	\br{h_{1}\ast h_{2}\ast\cdots	\ast h_{N}}\\
					&=	\br{\gamma_{0}^{-1}\ast h_{1}\ast\gamma_{1}\ast\gamma_{1}^{-1}\ast h_{2}\ast\gamma_{2}\ast\cdots\ast\gamma_{N-1}^{-1}\ast h_{N}\ast \gamma_{N}}\\
					&=	\br{g_{1}\ast g_{2}\ast\cdots	\ast g_{N}}\\
					&=	\br{g}
\end{align*}
proving the lemma.
\end{proof}

In the previous proof, the second collection of subbasic open sets in~\eqref{defnewV} are essential. Figure~\ref{counter} shows two loops $g$ and $h$ based at $x$ in the annulus $X=S^{1}\times I$.
\begin{figure}[h!]
	\centerline{\includegraphics{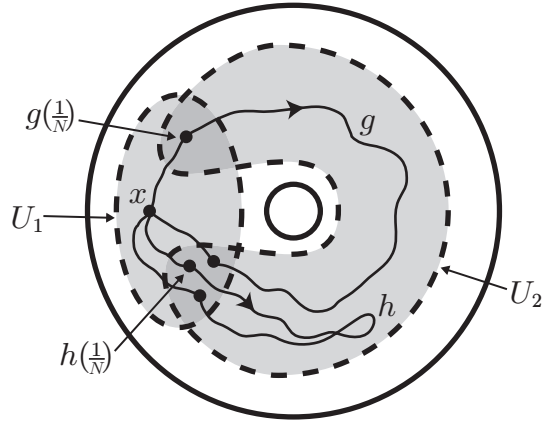}}
	\caption{Loops $g$ and $h$ based at $x$ in the annulus $X$.}
	\label{counter}
\end{figure}
All conditions in the proof are satisfied except $g(1/N)$ and $h(1/N)$ fail to lie in the same connected component of $U_{1}\cap U_{2}$. Clearly $g$ and $h$ are not homotopic loops.

\begin{proof}[Proof of Theorem~\ref{discretetheorem}]
First assume $\tfg{X}$ is discrete and let $x\in X$. By definition $\tfg{X,x}$ is discrete and so $\cbr{\br{e_{x}}}$ is open in $\tfg{X,x}$. By Lemma~\ref{lem1}, $x$ has a relatively inessential neighborhood in $X$. The choice of $x\in X$ was arbitrary and so $X$ is semilocally simply-connected.\\

Next assume $X$ is semilocally simply-connected and let $x\in X$. Points in $\tfg{X,x}$ are open by Lemma~\ref{lem2} and so $\tfg{X,x}$ is discrete. The choice of $x\in X$ was arbitrary and so $\tfg{X}$ is discrete.
\end{proof}

\section{Basepoint change}\label{basepoint}

\begin{lemma}
Let $X$ be a topological space and $x,y\in X$. If $x$ and $y$ lie in the same path component of $X$, then $\tfg{X,x}$ and $\tfg{X,y}$ are homeomorphic.
\end{lemma}

\begin{proof}
Let $\gamma:I\rightarrow X$ be a continuous path with $\gamma(0)=y$ and $\gamma(1)=x$. Define the function
\[
\xymatrix@R=0pt{
	C_{y}(X)	\ar[r]^-{\Gamma}	&		C_{x}(X)\\
	f					\ar@{|-{>}}[r]		&		\pa{\gamma^{-1}\ast f}\ast\gamma}
\]
First we show that $\Gamma$ is continuous. Let $I_{1}=\br{0,1/4}$, $I_{2}=\br{1/4,1/2}$, and $I_{3}=\br{1/2,1}$. Define the affine homeomorphisms
\[
\xymatrix@R=0pt{
	I_{1}	\ar[r]^-{s_{1}}		&				I			&	I_{2}	\ar[r]^-{s_{2}}		&		I					&	I_{3}	\ar[r]^-{s_{3}}		&		I\\
	t			\ar@{|-{>}}[r]		&				4t		&	t			\ar@{|-{>}}[r]		&		4t-1			&	t			\ar@{|-{>}}[r]		&		2t-1}
\]
and note that
\[
\xymatrix@R=0pt{
	I		\ar[r]^-{\Gamma(f)}	&		X\\
	t		\ar@{|-{>}}[r]			&		\gamma^{-1}	\circ s_{1}(t)	&	0\leq t\leq\frac{1}{4}\\
	t		\ar@{|-{>}}[r]			&		f						\circ s_{2}(t)	&	\frac{1}{4}\leq t\leq\frac{1}{2}\\
	t		\ar@{|-{>}}[r]			&		\gamma			\circ s_{3}(t)	&	\frac{1}{2}\leq t\leq1}
\]
Consider an arbitrary subbasic open set
\[
	V=V\pa{K,U}	\subset C_{x}(X).
\]
Observe that $\Gamma(f)\in V$ if and only if
\begin{align}
	\gamma^{-1}	&\circ s_{1}\pa{K\cap I_{1}}\subset U,	\label{ginvs}\\
	f						&\circ s_{2}\pa{K\cap I_{2}}\subset U,	\text{ and}	\label{fs}\\
	\gamma			&\circ s_{3}\pa{K\cap I_{3}}\subset U.	\label{gs}
\end{align}
Define the subbasic open set
\[
	V'=V\pa{s_{2}\pa{K\cap I_{2}},U}	\subset C_{y}(X).
\]
Observe that $f\in V'$ if and only if~\eqref{fs} holds. As conditions~\eqref{ginvs} and~\eqref{gs} are independent of $f$, either $\Gamma^{-1}(V)=\varnothing$ or $\Gamma^{-1}(V)=V'$. Thus $\Gamma$ is continuous. Next consider the diagram
\[
\xymatrix{
		C_{y}(X)	\ar[r]^-{\Gamma}	\ar[d]_{q_{y}}	&	C_{x}(X)	\ar[d]^{q_{x}}\\
		\tfg{X,y}	\ar@{-->}[r]^-{\pi\pa{\Gamma}}		&	\tfg{X,x}}
\]
The composition $q_{x}\circ\Gamma$ is constant on each fiber of $q_{y}$ so there is a unique set function making the diagram commute, namely $\pi\pa{\Gamma}:\br{f}\mapsto\br{\Gamma(f)}$. As $q_{y}$ is a quotient map, the universal property of quotient maps~\cite[Thm.~11.1~p.~139]{jM75} implies that $\pi\pa{\Gamma}$ is continuous. It is well known that $\pi\pa{\Gamma}$ is a bijection~\cite[Thm.~2.1~p.~327]{jM75}. Repeating the above argument with the roles of $x$ and $y$ interchanged and the roles of $\gamma$ and $\gamma^{-1}$ interchanged, we see that $\pi\pa{\Gamma}^{-1}$ is continuous. Thus $\pi\pa{\Gamma}$ is a homeomorphism as desired.
\end{proof}

\section{Translation}\label{translation}

\begin{lemma}\label{leftrightlemma}
Let $\pa{X,x}$ be a pointed topological space. If $\br{f}\in\tfg{X,x}$, then left and right translation by $\br{f}$ are self homeomorphisms of $\tfg{X,x}$.
\end{lemma}

\begin{proof}
Fix $\br{f}\in\tfg{X,x}$ and consider left translation by $\br{f}$ on $\tfg{X,x}$
\[
\xymatrix@R=0pt{
	\tfg{X,x}	\ar[r]^-{L_{\br{f}}}	&		\tfg{X,x}\\
	\br{g}		\ar@{|-{>}}[r]				&		\br{f}\cdot\br{g}}
\]
Plainly $L_{\br{f}}$ is a bijection of sets. Consider the commutative diagram
\begin{equation}\label{leftdiagram}
\xymatrix{
		C_{x}(X)	\ar[r]^-{L_{f}}	\ar[d]_{q}	&	C_{x}(X)	\ar[d]^{q}\\
		\tfg{X,x}	\ar@{-{>}}[r]^-{L_{\br{f}}}	&	\tfg{X,x}}
\end{equation}
where $L_{f}$ is defined by
\[
\xymatrix@R=0pt{
	C_{x}(X)	\ar[r]^-{L_{f}}	&		C_{x}(X)\\
	g					\ar@{|-{>}}[r]	&		f\ast g}
\]
First we show $L_{f}$ is continuous. Let $I_{1}=\br{0,1/2}$ and $I_{2}=\br{1/2,1}$. Define the affine homeomorphisms
\[
\xymatrix@R=0pt{
	I_{1}	\ar[r]^-{s_{1}}		&				I			&	I_{2}	\ar[r]^-{s_{2}}		&		I	\\
	t			\ar@{|-{>}}[r]		&				2t		&	t			\ar@{|-{>}}[r]		&		2t-1}
\]
and note that
\[
\xymatrix@R=0pt{
	I		\ar[r]^-{f\ast g}	&		X\\
	t		\ar@{|-{>}}[r]			&		f\circ s_{1}(t)	&	0\leq t\leq\frac{1}{2}\\
	t		\ar@{|-{>}}[r]			&		g\circ s_{2}(t)	&	\frac{1}{2}\leq t\leq1}
\]
Consider an arbitrary subbasic open set
\[
	V=V\pa{K,U}	\subset C_{x}(X).
\]
Observe that $f\ast g\in V$ if and only if
\begin{align}
	f	&\circ s_{1}\pa{K\cap I_{1}}\subset U	\text{ and}	\label{fs1}\\
	g	&\circ s_{2}\pa{K\cap I_{2}}\subset U.\label{gs1}
\end{align}
Define the subbasic open set
\[
	V'=V\pa{s_{2}\pa{K\cap I_{2}},U}	\subset C_{x}(X).
\]
Observe that $g\in V'$ if and only if~\eqref{gs1} holds. As condition~\eqref{fs1} is independent of $g$, either $L_{f}^{-1}(V)=\varnothing$ or $L_{f}^{-1}(V)=V'$. Thus $L_{f}$ is continuous. The composition $q\circ L_{f}$ is constant on each fiber of the quotient map $q$ and~\eqref{leftdiagram} commutes, so the universal property of quotient maps~\cite[Thm.~11.1~p.~139]{jM75} implies that $L_{\br{f}}$ is continuous.\\

Applying the previous argument to $f^{-1}$ we get $L_{\br{f}}^{-1}=L_{\br{f^{-1}}}$ is continuous and $L_{\br{f}}$ is a homeomorphism. The proof for right translation is almost identical.
\end{proof}

\end{document}